\documentclass{article}								
\usepackage{amsmath,amssymb,amsthm}				
\usepackage{verbatim}								
\usepackage{amscd}								
\usepackage[usenames]{color}								
\input xy									
\xyoption{all}									
\newtheorem{Thm}{Theorem}[section]
\newtheorem{Cor}[Thm]{Corollary}						
\newtheorem{Prop}[Thm]{Proposition}					
\newtheorem{Lemma}[Thm]{Lemma}

\theoremstyle{definition}

\newtheorem{Def}[Thm]{Definition}						
\numberwithin{equation}{section}								
\def\q{\mathbb{Q}}								
\def\r{\mathbb{R}}								
\def\z{\mathbb{Z}}

\def\h{\mathbb{H}}								
								
\newcommand\inv{^{-1}}								
\newcommand {\ignore}[1]  {}

\def\lg{\Lambda(\Gamma)}	
								
\begin{document}

\begin{center} Limit Points Badly Approximable by Horoballs
\end{center}
\begin{center}
Dustin Mayeda\\  Department of Mathematics, UC Davis, Davis CA 95616, USA\\
{\it dmayeda@math.ucdavis.edu}

\vspace{.25 in} 

Keith Merrill\\ Department of Mathematics, Brandeis University, Waltham MA 02454, USA\\
{\it merrill2@brandeis.edu}
\end{center}

\begin{abstract} For a proper, geodesic, Gromov hyperbolic metric space $X$, a discrete subgroup of isometries $\Gamma$ whose limit set is uniformly perfect, and a disjoint collection of horoballs $\{ H_j \}$, we show that the set of limit points badly approximable by $\{ H_j \}$ is absolutely winning in the limit set $\Lambda(\Gamma)$. As an application, we deduce that for a geometrically finite Kleinian group acting on $\h^{n+1}$, the limit points badly approximable by parabolics, denoted $BA(\Gamma)$, is absolutely winning, generalizing previous results of Dani and McMullen. As a consequence of winning, we show that $BA(\Gamma)$ has dimension equal to the critical exponent of the group. Since $BA(\Gamma)$ can alternatively be described as the limit points representing bounded geodesics in the quotient $\h^{n+1}/\Gamma$, we recapture a result originally due to Bishop and Jones.
\end{abstract}

\begin{section}{Introduction}
\begin{subsection}{Background and Main Results}
The deep connections between Diophantine approximation and hyperbolic geometry have been a subject of active interest for many years (see \cite{HP} and references therein). Classically, one principal object of study is the set of {\it badly approximable} numbers, $$BA = \left\{ x \in \r: \exists C > 0 \text{ such that } \left| x - \frac{p}{q}\right| > \frac{C}{q^2} \text{ for all } \frac{p}{q} \in \q\right\}.$$ This notion of approximation is of interest in light of Dirichlet's Theorem, which states that the opposite inequality holds for every point if $C$ is replaced by $1$. Historically, Louiville first shows that $BA$ is nonempty as a consequence of his theorem on approximation of algebraic numbers, in particular irrational quadratic numbers are badly approximable. Using a characterization of $BA$ via continued fractions, Borel and Bernstein showed that the Lebesgue measure of $BA$ is $0$. Nonetheless, it has Hausdorff dimension $1$ (\cite{J}) and it was known to Davenport that its intersection under finitely many translations is nonempty (see \cite{B} for an extensive historical discussion of $BA$).

Then in \cite{Sch}, Schmidt introduces a two-player game, henceforth called Schmidt's game, whose winning sets enjoy those and stronger properties (see $\S 2$ for definitions). In particular, Schmidt shows that winning sets in $\r$ are dense, have full dimension, and that the intersection with their images under countably many $C^1$ diffeomorphisms of $\r$ is nonempty. 

\ignore{
Because of the remarkable rigidity winning sets enjoy, Schmidt's game has been extended to a variety of contexts where winning sets exhibit properties like the ones above. The diversity of these settings can be seen in the following applications. In \cite{CCM} it is shown that the set of directions for which the Teichmuller flow remains bounded is winning, \cite{D} for bounded geodesics on homogeneous spaces, \cite{F} on fractals, \cite{BFK} for toral endorphisms avoiding an open subset of $T^n$, and \cite{ET} for winningness of badly approximable affine forms. In \cite{McM}, McMullen proposes a strengthening of the game whose winning subsets, said to be {\it absolutely winning}, enjoy even stronger rigidity properties than those mentioned above. In particular, absolutely winning sets are winning for the original game. \cite{KW} also...
(HOW DO WE HANDLE THIS SECTION? JUST REFERENCES?)
}

Because of the remarkable rigidity winning sets enjoy, Schmidt's game has been extended to a variety of contexts. The diversity of these settings can be seen in the following papers: \cite{CCM}, \cite{D}, \cite{F}, \cite{BFK}, \cite{ET}, \cite{KW}, and \cite{McM} (see $\S 2$ for a more detailed discussion of some of these results which will be needed in what follows). In particular, in \cite{McM}, McMullen defines a strengthening of the original game, called the absolute game, whose winning sets are said to be {\it absolutely winning} and which enjoy even stronger rigidity properties than winning sets for Schmidt's original game. 

In this paper we play McMullen's variant of the game on the limit set of a discrete group $\Gamma$ of isometries of a proper Gromov hyperbolic space $X$. We generalize the above notion of BA by considering the set of limit points which are badly approximable by a disjoint collection of horoballs. Our main result is as follows:
\begin{Thm}\label{main}
Let $X$ be a proper, geodesic, Gromov hyperbolic space, $\Gamma \subset \operatorname{Isom}(X)$ a non-elementary discrete subgroup such that $\Lambda(\Gamma)$ is uniformly perfect, and $\{ H_j \}$ a disjoint collection of horoballs. Then the set of limit points which are badly approximable with respect to $\{ H_j \}$ is absolutely winning in $\Lambda(\Gamma)$. 
\end{Thm}

Of particular interest is the case where $X = \h^{n+1}$ and $\Gamma$ is a Kleinian group. We consider a $\Gamma$-invariant collection of horoballs based at the parabolic fixed points of $\Gamma$, generalizing the notion of approximation by parabolics. The set of such badly approximable limit points, denoted $BA(\Gamma)$, has been extensively studied by Patterson, Stratmann, Hill, and Velani in the case that $\Gamma$ is geometrically finite (\cite{P}, \cite{S}, \cite{SV}, \cite{HV}, and \cite{BDV}), and in particular, analogues of theorems of Dirichlet, Khintchine, and Jarnik-Besicovitch are known to hold. The primary motivation for this is that if we take $\Gamma = PSL(2,\z)$, then the set of points badly approximable by $\Gamma$ is exactly the set $BA$ introduced above.

In the case of a lattice $\Gamma$ acting on $\h^{n+1}$, it was shown by Dani \cite{D} that $BA(\Gamma)$ is winning in $\r^n$, and by McMullen \cite{McM} that it is absolutely winning with respect to the strengthening of the game he defines. In this conext our theorem can be interpreted as a further generalization of these results.
\begin{Cor}\label{main2}
Let $\Gamma \subset \operatorname{Isom}(\h^{n+1})$ be a non-elementary finitely generated Kleinian group such that there exists a $\Gamma$-invariant disjoint collection of horoballs based at the parabolic fixed points of $\Gamma$. Then $BA(\Gamma)$ is absolutely winning in $\lg$. This holds in particular for geometrically finite $\Gamma$.
\end{Cor}
Moreover, in the geometrically finite context we have an identification, implicitly contained in \cite{Bow}, that $BA(\Gamma)$ corresponds precisely to the set of endpoints of geodesic rays whose projection to the quotient manifold $\h^{n+1}/\Gamma$ remain bounded. Thus ergodicity of the geodesic flow with respect to the Patterson-Sullivan measure class implies that $BA(\Gamma)$ has measure $0$. Nonetheless, using a lemma communicated to us by Fishman, Simmons, and Urbanski, we show that the dimension of $BA(\Gamma)$ equals the critical exponent of $\Gamma$, thus recapturing a result of Bishop and Jones \cite{BJ}.

Finally, using the fact that the product of winning sets is again winning and the diffuseness of $\lg$, we show that the set of bounded geodesics in $\h^{n+1}/\Gamma$ is winning in $\lg \times \lg$. 
\end{subsection}

\begin{subsection}{Outline of the paper}

\indent In $\S 2$ we define the necessary definitions and lemmas needed to state and prove the main result.

In $\S 3$ we prove Theorem \ref{main}. The proof is broken into three steps. First we generalize a lemma from \cite{McM}, which states that the set of points in the boundary $\partial X$ which avoid the shadows of a small rescaling of the collection $\{ H_j \}$ is absolutely winning in $\partial X$. Next we show that any such boundary point is badly approximable by $\{ H_j \}$. Finally we recall a certain geometric property of the limit set $\lg$ which implies that absolute winningness is preserved under intersection. Therefore, the set of limit points badly approximable by $\{ H_j \}$ is absolutely winning in $\lg$. 

In $\S 4$, we specialize to the case $X = \h^{n+1}$, and investigate conditions under which the limit set $\lg$ is uniformly perfect and such a collection $\{ H_j \}$ exists. In particular, we show that for geometrically finite Kleinian groups both of these conditions hold. Therefore, Corollary \ref{main2} follows. 

Finally in $\S 5$ we prove that in the Kleinian group case, absolutely winning subsets of $\lg$ have dimension $\delta$, the critical exponent of the group and the dimension of the conical limit set. Thus we reprove the main result of \cite{BJ} in the case that $\Gamma$ is finitely generated and $\{ H_j \}$ exists. As noted previously, this holds in particular for geometrically finite Kleinian groups.

\end{subsection}

\begin{subsection}{Acknowledgements}
The authors would like to thank Dmitry Kleinbock for his patience and many helpful revisions, Ryan Broderick for his expertise in arguments about properties of winning sets, and Lior Fishman, David Simmons, and Mariusz Urbanski for furnishing the full dimension of winning sets. Finally, we would like to thank the referee for his many invaluable comments and simplifications of the main results.
\end{subsection}

\end{section}

\begin{section}{Definitions}
\begin{subsection}{Schmidt's Game}

In \cite{Sch}, Schmidt introduced and analyzed the following game between two players, whom we will call Alice and Bob. Let $(X,d)$ be a complete metric space, and fix parameters $0 < \alpha, \beta < 1$. Set $\Omega = X \times \r_+$, and define a partial ordering on $\Omega$ as follows: $$(x_2,\rho_2) \leq (x_1,\rho_1) \text{ if and only if } \rho_2 + d(x_1,x_2) \leq \rho_1.$$ Then Alice and Bob alternately choose pairs $w_i' = (x_i', \rho_i')$ and $w_i = (x_i, \rho_i),$ respectively, satisfying $$w_i' \leq w_i \text{ and } w_{i+1} \leq w_i'$$ with $$\rho_i' = \alpha \rho_i \text{ and } \rho_{i+1} = \beta \rho_i'$$  We may think of the pair $(x,\rho)$ as determining a ball $B(x,\rho) \subset (X,d)$, in which case the partial ordering $(x_1,\rho_1) \leq (x_2,\rho_2)$ implies that $B(x_2,\rho_2) \subset B(x_1,\rho_1)$ (but in fact can be strictly stronger for general metric spaces). Let us denote by $B_i$ the ball $B(x_i',\rho_i')$ and $A_i$ the ball $B(x_i,\rho_i)$. Since the balls are closed, nested, and $X$ is complete, the intersection $\cap_i B_i$ is a point $x_\infty \in X$. We say that a set $S \subset X$ is $(\alpha,\beta)$-{\it winning} if Alice has a strategy such that, no matter how Bob plays, $$ \bigcap_i B_i \cap S \ne \emptyset.$$ We call $S$ $\alpha$-{\it winning} if it is $(\alpha,\beta)$-winning for every $\beta > 0$, and is called {\it winning} if it is $\alpha$-winning for some $\alpha$. 

It was first noted in \cite{F} that when playing the game on a metric space $(X,d)$, it makes to sense to play the game on any closed subset $E \subset X$ by requiring Bob and Alice's balls to be centered on points of $E$. Since $E$ is closed, the unique point of intersection $\cap_i B_i$ will lie in $E$. It is not guaranteed, however, that simply because $S \subset X$ is winning, that $S \cap E$ is winning in $E$. 

As mentioned in the introduction, winning sets in many contexts exhibit strong properties like having full dimension, and the class of winning sets is closed under  bi-Lipschitz mappings, and perhaps most strikingly, countable intersections. This latter condition is used in \cite{Sch} to show that the set of real numbers normal to no base is a winning set, and hence has full dimension. Following \cite{F}, in \cite{BFKRW} the authors show that the set of such numbers belonging to any fractal supporting a "nice" measure, for instance the Cantor middle-third set, has full dimension in the fractal. It is this idea of supporting reasonable measures which allows the rigidity properties of winning sets to hold, and hence in turn to the diverse applications the game enjoys. Notably, Dani shows in \cite{D} that the set of bounded orbits on certain homogeneous spaces, while typically measure $0$, is nonetheless winning, and hence has full dimension. Similarly, it is shown in \cite{CCM} that the set of directions under which the Teichmuller flow remains bounded is winning. And in \cite{KW}, Schmidt's game, and the countable intersection property of winning sets, is used to prove the abundance of non-quasiunipotent orbits whose orbit closures avoid a countable collection of points, thus settling a conjecture of Margulis, \cite{Ma}.

In \cite{McM}, McMullen proposed the following variant of Schmidt's game, in which instead of choosing where Bob must play, Alice now chooses where he cannot: $$B_1 \supset B_1 - A_1 \supset B_2 \supset B_2 - A_2 \supset ...$$ subject to $$R(B_{i+1}) \geq \beta R(B_i), R(A_i) \leq \beta R(B_i),$$ with $\beta < 1/3$ fixed (McMullen defines the game for balls in $\r^n$, but the definition clearly makes sense in general).  Note that we define this variant by referring directly to closed balls in $X$, and not a partial ordering on $\Omega$--see the proof of Lemma \ref{abswin} below for a discussion on this. Similarly to before, we call a subset $S \subset X$ {\it absolutely winning} if there exists $\beta_0 \in (0,1/3)$ such that for every $\beta \leq \beta_0$, Alice has a strategy which guarantees $$\bigcap_i B_i \cap S \ne \emptyset$$ for the game with parameter $\beta$.  Note in this case that $\cap_i B_i$ may fail to be a point.

The reason for the stipulation $\beta < 1/3$ is to prevent Alice's strategy from leaving Bob with no legal plays. For instance, if $\beta \geq 1/3$ and Alice chooses the ball $A_i$ concentric with $B_i$ and with radius scaled by $\beta$, then there is no nested ball $B_{i+1}$ satisfying $$B_{i+1} \subset B_i - A_i ,$$ terminating the game prematurely. When we look to define the absolute game on a closed subset $E \subset X$ analogously to before, this problem becomes even more pronounced. That is, care must be taken to ensure that no choice of $A_i$ leaves Bob no legal plays $B_{i+1}$. 

The property needed to ensure that this cannot happen is discussed in \cite{BFKRW}, and is called diffuseness (in the terminology of \cite{BFKRW}, we will be considering $0$-dimensionally diffuse sets and $0$-dimensionally absolutely winning sets). Specifically, we call a closed subset $E$ of $X$ $\beta_0$-{\it diffuse} if for every $x \in E$ and $\rho$ sufficiently small, and for any ball $B(y,\beta_0 \rho)$ with $y \in X$, there exists $x' \in E$ such that $$B(x',\beta_0 \rho) \subset B(x,\rho) - B(y,\beta_0 \rho).$$ Then for the $\beta$-absolute game with $\beta \leq \beta_0$, no matter what Euclidean ball $B(y,\beta \rho)$ Alice's strategy removes, Bob always has at least one legal move. As usual, a set will be called {\it diffuse} if it is $\beta_0$-diffuse for some $\beta_0$ (Note that our terminology differs slightly from that used in \cite{BFKRW}, where our use of diffuseness is the consequence of their Lemma 4.3).

For diffuse sets we have the following, Proposition $4.9$ in \cite{BFKRW}:
\begin{Thm}\label{diffuse}
If $S \subset X$ is absolutely winning in $X$, then for any diffuse subset $E \subset X$, we have that $E \cap S$ is absolutely winning in $E$.
\end{Thm}

We will also need the following lemmas from \cite{McM}:

\begin{Lemma}\label{abswin}
 Every absolutely winning set is winning.
\end{Lemma} 
\begin{proof} Let $B_i = B(x_i,\rho_i)$ denote Bob's $i^{th}$ choice as before, and $A_i = B(y_i, r_i)$ Alice's play as per the strategy, so that  $r_i \leq \beta \rho_i$. Then by definition of absolutely winning, there is at least one point $x_{i+1}$ such that $$B(x_{i+1},\beta \rho_i) \subset B(x_i,\rho_i) - B(y_i,r_i).$$ But then it follows that $w_i' = (x_{i+1}, \beta \rho_i) \leq w_i = (x_i, \rho_i)$. So setting Alice's move in the original game as $B(x_{i+1},\beta \rho_i)$ shows that $E$ is in fact $\beta$-winning. 
\end{proof}

\begin{Lemma} \label{intersection}
The countable intersection of absolutely winning sets is again absolutely winning.
\end{Lemma}

This latter lemma highlights one of the primary rigidity properties both winning and absolutely winning sets exhibit, and one which will be essential in $\S 4$. 

We now recall the following definition for metric spaces. We say that a metric space $(X,d)$ is {\it uniformly perfect} if there exists a constant $0 < \nu < 1$ such that for any metric ball $B = B(x,R)$ such that $X - B \ne \emptyset$, we have $B(x,R) - B(x,\nu R) \ne \emptyset$. We call a subset $E \subset X$ {\it uniformly perfect} if it is uniformly perfect as a metric space with the induced metric. 

The following lemma shows the relation between uniform perfectness and the notion of diffuseness discussed above:

\begin{Lemma}\label{diffuse is u.p.}
 Let $E \subset \r^n$ be a uniformly perfect subset with associated parameter $\nu$. Then for any $$\beta < \min\left\{1-\nu,\frac{\nu^2}{4}\right\},$$ $E$ is $\beta$-diffuse.
\end{Lemma} 
\begin{proof} Fix $x \in E$, and let $\rho > 0$ be sufficiently small so that $E \not \subset B(x,\rho)$. Let $y \in \r^n$ be arbitrary. We break the proof into two cases. 
\begin{enumerate}
\item If $d(x, y) > 2\beta\rho$, then it is clear that $B(x,\beta\rho)$ has the requisite property, i.e. we can take $x=x'$.
\item If $d(x,y) \leq 2\beta \rho$, then it suffices to produce a point $x' \in E$ such that $4\beta \rho < d(x,x') < (1-\beta)\rho$, since then $B(x',\beta \rho)$ works. But by uniform perfectness, since $E - B(x,\rho) \ne \emptyset$, for every $n$, we can find $$x_n \in E \mbox{ such that } \nu^n \rho < d(x,x_n) <\nu^{n-1}\rho$$
For our choice of $\beta$, taking $n=2$ we have $$4\beta \rho < \nu^2 \rho < d(x,x_2) < \nu \rho < (1-\beta)\rho$$ So we let $x' = x_2 \in E$ and use the ball $B(x',\beta \rho)$. 
\end{enumerate}
\end{proof}
\end{subsection}

\begin{subsection}{Gromov Hyperbolic Spaces}
Let $X$ be a metric space and recall that $X$ is called {\it proper} if all closed balls are compact. By a {\it geodesic} in $X$ we mean an isometric embedding of some subset $[a,b] \subset \r$ into $X$.  We say that $X$ is {\it geodesic} if every pair of points in $X$ is connected by a geodesic.  For $p, q \in X$ a geodesic space we let $[p,q]$ denote a geodesic from $p$ to $q$, even though the geodesic may not be unique.  Now we call $X$ $\delta$-{\it hyperbolic} if $X$ is geodesic and for any three points $a, b, c \in X$ we have that $$[b,c] \subset N_\delta([a,b]) \cup N_\delta([a,c]),$$ where $N_\delta(A)$ for any subset $A$ of $X$ is the $\delta$-neighborhood of A.  

For the rest of the paper $X$ will denote a proper geodesic $\delta$-hyperbolic metric space and $o$ a fixed basepoint in $X$. Define an equivalence relation $\sim$ on the set of geodesic rays (i.e. isometric embeddings of $[0,\infty)$) as follows: $\gamma \sim \sigma$ if and only if 

\begin{center}  $d(\gamma(t), \sigma(t)) < C$ for some $C > 0$ for all $t$ sufficiently large.
\end{center}

The {\it boundary} of $X$, denoted by $\partial X$, is defined as the set of equivalence classes of the above relation. The {\it Gromov product} of two points $x, y \in X$ with respect to the basepoint $o$ is defined as 

\begin{center}  $(x|y):=\frac{1}{2}(d(x,o) + d(y,o) - d(x,y))$
\end{center}

In general, the Gromov product measures the failure of the triangle inequality to be an equality. In particular, if $X$ is a tree then $(x|y)$ is the length of the overlap between the unique geodesics from $o$ to $x$ and from $o$ to $y$.  The Gromov product can be extended to the boundary by considering the following equivalent definition of $\partial X$.  A sequence $(x_i)$ of points in $X$ is called {\it convergent at infinity} if

\begin{center}  $\displaystyle\lim_{i,j\to\infty}(x_i|x_j) = \infty$.
\end{center}

We say that two sequences $(x_i)$ and $(y_i)$ are equivalent if $\lim_{i,j\to\infty}(x_i|y_j) = \infty$.  The boundary is then defined to be the set of equivalence classes and the Gromov product on the boundary is defined as $$(\xi|\nu):=\inf \displaystyle\liminf_{i,j\to\infty}(x_i|y_j) \text{ for } \xi,\nu \in \partial X$$ where the infimum is taken over all sequences $(x_i)$ and $(y_j)$ which represent $\xi$ and $\nu$ respectively.

The boundary $\partial X$ supports a family of metrics called {\it visual metrics} which induce the same topology.  A visual metric is any metric $\rho$ which satisfies 
\begin{center}
$\frac{1}{C}a^{-(x|y)} < \rho(x,y) < Ca^{-(x|y)}$ for some $C > 1$ and $a > 1$ and all $x,y \in \partial X.$
\end{center}
See Proposition 3.21 in chapter III.H of \cite{BH} for the existence of visual metrics. 

 For a geodesic ray $\gamma$ define the {\it Busemann function} $b_\gamma(z) = \lim_{t\to\infty}(d(\gamma(t),z) - t) $ where $z \in X$.  A {\it horoball} $H$ based at $\gamma(\infty) \in \partial X$ is a subset of $X$ of the form $H = b_\gamma^{-1}(-\infty, s]$ for some number $s$.  We define the {\it shadow} of a subset $A \subset X$, denoted by $Sh(A)$, to be the collection of endpoints of geodesic rays starting at $o$ which intersect A.  Let $H$ be a horoball based at $\xi$ and let $R_\xi > 0$ be the {\it radius} of $H$, that is the infimum such that $Sh(H) \subset B(\xi,R_\xi)$.  For $0 < s <1$ we define the {\it scaling} of $H$ by $s$, denoted $sH$, to be the union of horoballs whose shadows are contained in $B(\xi,sR_\xi)$.  To a ball $B(\gamma(\infty),r) \subset \partial X$, we assign the horoball $H_{\gamma(\infty)} = b_{\gamma}^{-1}(-\infty, \log r]$. 

We now give the main definition of the paper:
\begin{Def} Let $\{ H_j \}$ be a collection of horoballs, based at $\{ \xi_j \}$ with radius $\{R_j \}$. We say that a boundary point $x \in \partial X$ is {\it badly approximable} by $\{ H_j \}$ if there exists $s > 0$ such that $\rho(x,\xi_j)  > s \cdot R_j$ for every $j$. 
\end{Def} 

Let $\Gamma$ be a group of isometries of $X$ which act properly discontinuously on $X$, i.e. the set $\{\gamma \in \Gamma: B(x,r) \cap \gamma B(x,r) \neq \emptyset\}$ is finite for all $x \in X$ and $r > 0$.  Note that $\Gamma$ therefore acts on $\partial X$ by $\gamma \cdot (x_i) = (\gamma \cdot x_i)$ and define the limit set $\lg \subset \partial X$ of the action as

\begin{center}  $\lg = \{(x_i) \in \partial X | x_i = \gamma_i \cdot o, \gamma_i \in \Gamma, 1 \leq i < \infty\}.$
\end{center}

It can be shown that $\lg$ consists of either $0, 1, 2$, or infinitely many points.  If $\lg$ contains infinitely many points then we say that $\Gamma$ is {\it nonelementary}.

We have mentioned that in general geodesics may not be unique, but by the following proposition, such a choice of geodesic will not affect our results. 

\begin{Prop}(Stability of geodesics) \label{stability}
Let $X$ be a $\delta$-hyperbolic space and suppose that $\gamma$ and $\sigma$ are two geodesics which share the same endpoints.  Then we have $\gamma \subset N_{2\delta}(\sigma)$ and $\sigma \subset N_{2\delta}(\gamma)$.
\end{Prop}

The following tool from \cite{GH} is extremely useful for working in hyperbolic spaces. Recall that a geodesic metric space is called a {\it real tree} if the union of two geodesic segments with exactly one endpoint in common is again a geodesic segment.
\begin{Thm} (Tree Approximation) \label{tree}
Let X be a $\delta$-hyperbolic space and let F be a collection of n geodesic rays starting at $o$ and $k$ an integer such that $n \leq2^k$.  Then there is a real tree $T$ with base point $t$ and a map $\phi :F \to T$ preserving the distance to the base point such that

\begin{center} $ d(x,y) - 2(k+1) \delta \leq d(\phi(x),\phi(y)) \leq d(x,y)$
\end{center}
for all $x,y \in F$.
\end{Thm}
\end{subsection}

As an example of how tree approximation functions, we will explicitly characterize the above concepts in the case of a tree. In this case, geodesics are unique, and as noted before the Gromov product $(x|y)$ is equal to the length of overlap between the unique geodesics $[o,x]$ and $[o,y]$. Let $T$ be a real tree, $H = \beta\inv_{\gamma}(-\infty,-s]$, and $\gamma(t)$ the geodesic representing $\gamma$. Then $$x \in H \Leftrightarrow \beta_{\gamma}(x) < -s  \Leftrightarrow \lim_{t} d(x,\gamma(t)) - t < -s.$$ In this case, we can succintly compute this limit:
\begin{eqnarray*}
-s &>& \lim_t d(x,\gamma(t)) - t\\ &=&\lim_t  d(o,x) + t - 2(x|\gamma(t)) - t\\ &=& \lim_t d(o,x) - 2(x|\gamma(t))\\ &=& d(o,x) - 2(x|\gamma).
\end{eqnarray*}
However, any such $x$ must have $d(o,x) \geq s$, and hence $H \subset \{ x \in T : (x|\gamma) > s\}$. Note that if $\eta \in Sh(H)$, then there exists $t$ such that $\eta(t) \in H$. Then by our previous remarks, we have $$(\eta|\gamma) \geq (\eta(t)|\gamma) \geq s.$$ Therefore, if $\eta \in Sh(H)$, $\rho(\eta,\gamma) = e^{-(\eta|\gamma)} \leq e^{-(\eta(t)|\gamma)} \leq  e^{-s}.$ In particular, the shadow of the horoball based at $\xi$ corresponding to $\log R$ is contained in $B(\xi,R)$. It follows from this result and tree approximation that in the general case, the shadow of the horoball based at $\xi$ corresponding to $\log_a R$ is contained in $B(\xi, C e^{4\delta}R)$, where $C$ is as in the definition of a visual metric.
\end{section}

\begin{section}{Proof of Main Theorem}

For each point $x \in \partial X$ let $\gamma_x$ denote a geodesic ray from $o$ to $x$. Note that while such a geodesic will in general not be unique, the choice of the geodesic will not affect our results by stability of geodesics (Prop \ref{stability}) since if we chose another such geodesic ray, our calculations would only change by at most $2\delta$. Fix a visual metric with visual constant $a$, and denote by $\log := \log_a$.

\begin{Thm}
Let $X$ be a proper geodesic $\delta$-hyperbolic space and $\{ H_j \}$ a countable collection of disjoint horoballs.  Then the set of boundary points badly approximable by $\{H_j \}$ is $\beta$-absolutely winning with respect to any visual metric for $\beta < \frac{1}{3}$.
\end{Thm}

\begin{proof}

First note that we can replace the collection $\{ H_j \}$ by any rescaling of the same collection, without affecting the result. In particular, we may assume that the basepoint $o$ is not contained in any horoball $H \in \{ H_j \}$. 

We then define a strategy as follows:  Let $B_i = B(x_i,\rho_i) \subset \partial X$ denote Bob's $i^{th}$ move.  If $\gamma_{x_i} (-\log\rho_i)$ is in one of the horoballs $H$, which is unique by disjointness and which we will therefore denote $H_i$, then set $A_i = B(\xi_i,\beta\rho_i)$, where $H_i$ is based at $\xi_i \in \partial X$. Otherwise, choose $A_i$ disjoint from $B_i$.  Let $t_i = -\log(\rho_i)$ and suppose that $\cap_i B_i = \{x\}$.

We will need the following lemma:

\begin{Lemma}

If there exists $t > t_1$ such that $\gamma_x(t) \in cH$ for some $H \in \{ H_j \},$ where $c$ is a constant depending only on $a, \beta$, and $\delta,$ then there exists an $i \geq 1$ such that $\gamma_{x_i}(t_i) \in H$.  Moreover if $i > 1$ then $c \cdot diam(Sh(H)) \leq \rho_i$.
\end{Lemma}

\begin{proof}

Set $c = \frac{\beta}{C \cdot 4\delta \cdot a}e^{-\delta}$ and consider $\gamma_x^{-1}(\frac{1}{a}e^{-\delta}H) \subset (b,d)$ where $b$ is the first time $\gamma_x$ enters $\frac{1}{a}e^{-\delta}H$ and $d$ is the last time $\gamma_x$ exits $\frac{1}{a}e^{-\delta}H$.  If $t_i \in (b,d)$ then $\gamma_x (t_i)$ is within distance $\delta$ of $\frac{1}{a}e^{-\delta}H$ since horoballs are $\delta$-quasiconvex, i.e. for any two points in $H$, the geodesic segment between them lies in the $\delta$ neighborhood of $H$.  Thus $\gamma_x (t_i) \in \frac{1}{a}H$. We will denote $H$ by $H_i$ and set $R_i$ to be the radius of the shadow as before.

Now if $T$ is a real tree then $\rho(\xi,\nu) = e^{-(\xi|\nu)}$ is a visual metric on $\partial T$.  Hence if $\rho(x,x_i) \leq \rho_i$ then $(x|x_i) \geq t_i$ and so $\gamma_x$ and $\gamma_{x_i}$ coincide for at least length $t_i$ and thus $d(\gamma_x (t_i),\gamma_{x_i} (t_i)) = 0$  in  $T$.  Therefore by tree approximation (Prop \ref{tree}), $d(\gamma_x (t_i),\gamma_{x_i} (t_i)) \leq 4\delta$ in $X$.

Let $H$ be determined by $\gamma$ and let $z_1,z_2 \in \gamma$ with $z_1 \in \partial H$ and $z_2 \in \partial(cH)$.  Then observe that

\begin{center}  $\log(\frac{1}{c}) = \log(r) - \log(cr) = b_\gamma (z_1) - b_\gamma (z_2) = \lim_{t\to\infty}(d(\gamma(t),z_1) - d(\gamma(t),z_2)) = d(z_1,z_2) = d(\partial H,\partial(cH))$.
\end{center}

Thus $|b-d| \geq d(\partial(\frac{1}{a}H_i),\partial(\frac{1}{a}\beta e^{-\delta} H_i)) = \log(\frac{1}{\beta e^{-\delta}}) = -\log(\beta) + \delta$. On the other hand, $|t_{i+1} - t_i| = -\log(\frac{\rho_{i+1}}{\rho_i})| \leq -\log(\beta)$.  Hence $t_i - b \leq -\log(\beta) + \delta$.  Again if we are in a real tree $T$ then $b \leq -\log(R_i) + \log(a)$ since $\gamma_x$ intersects $\frac{1}{a}H_i$, and hence $x \in Sh(\frac{1}{a}H_i) \subset Sh(H_i) \subset B(\xi_i,R_i)$. Thus in the tree we get $b \leq -\log R_i + \log a$, and therefore by Prop \ref{tree}, in the general case, $R_i e^{b} \leq C \cdot 4\delta \cdot a$. Combining these estimates, we have $$\frac{R_i}{C \cdot 4\delta \cdot a} \leq e^{-b} \leq \frac{\rho_i}{\beta}e^{\delta},$$ and therefore, $$\frac{\beta}{C \cdot 4 \delta \cdot a} e^{-\delta} R_i \leq \rho_i,$$ as claimed.

\end{proof}

Returning to the strategy, consider the collection $\{ cH \}$, with $c$ as in the previous lemma. If $\gamma_x$ is disjoint from $cH$, then $\rho(x,\xi) > c \cdot e^{-4\delta} \frac{R_i}{2 C}$. Otherwise, by the lemma, there exists an $i$ such that $\gamma_{x_i}(t_i) \in H$, in which case Alice's move is $A_i = B(\xi_i, \beta \rho_i)$. Hence by definition $x \notin B(\xi_i, \beta \rho_i)$, and therefore $\rho(x,\xi_i) > \frac{\beta \rho_i}{R_i} R_i$. 

So the proof reduces to the claim that $s = \inf_i \left\{\frac{c}{C \cdot 2} e^{-\delta}, \frac{\beta \rho_i}{R_i}\right\} > 0$. But by the previous lemma, this infimum is bounded from below by $\inf\left\{ \frac{ c}{C \cdot 2} e^{-4\delta}, \frac{\beta \rho_1}{R_1}, \frac{\beta^2}{C \cdot 4 \delta \cdot a} e^{-\delta}\right\} > 0$. 

\end{proof}

To finish the proof of Theorem \ref{main}, we need to play the game on $\lg$, where $\Gamma \subset \operatorname{Isom}(X)$. So we need to know that the absolute winningness of the set of boundary points badly approximable by $\{ H_j \}$ is inherited by limit points of $\Gamma$. However, since $\lg$ is assumed uniformly perfect, by Lemma \ref{diffuse is u.p.} it is diffuse, and thus the set of limit points badly approximable by $\{ H_j \}$ is absolutely winning in $\lg$ by Theorem \ref{diffuse}. 

\end{section}

\begin{section}{Applications to Kleinian Groups}

We denote by $\h^{n+1}$ the simply connected Riemannian manifold with constant negative sectional curvature equal to $-1$, which is unique up to isometry.  In this paper we will work with the ball model of hyperbolic space. This model is given by $\mathcal{B} = \{x \in \r^{n+1} : |x| < 1\}$ with the metric given by $ds =\frac{|dx|}{(1 - |x|^2)}$. The geodesics correspond to radial lines through the origin and arcs of circles which intersect the unit sphere orthogonally. A {\it horoball} is a Euclidean ball which is internally tangent to the unit sphere, which is the boundary in this model.

A {\it  Kleinian group} $\Gamma$ is a discrete subgroup of $\operatorname{Isom}(\h^{n+1})$.  Each isometry uniquely extends to a conformal homeomorphism of the boundary $\partial \h^{n+1}$ and conversely each conformal homeomorphism of $\partial \h^{n+1}$ uniquely extends to an isometry of $\h^{n+1}$. The isometries can be classified into three types which are elliptic (fixes one point in $\h^{n+1}$), parabolic (fixes exactly one point in $\partial \h^{n+1}$ and no points in $\h^{n+1}$), and hyperbolic (fixes exactly two points in $\partial \h^{n+1}$ and no points in $\h^{n+1}$).  

In this section we specialize Theorem \ref{main} to the case of a Kleinian group $\Gamma$ acting on $\h^{n+1}$. In this context the notion of approximation by horoballs takes on a particularly important interpretation: we call a limit point $p \in \lg$ {\it parabolic} if it is the fixed point of a parabolic isometry in $\Gamma$. We wish to approximate limit points by parabolic limit points. As mentioned in the introduction, this notion of appoximation has been extensively studied, particularly because of its generalization of the classical BA. If we take $\Gamma = PSL_2(\z)$ acting on $\h^2$, then the parabolic fixed points are precisely the rationals and the point at infinity, so that approximation by parabolics captures the traditional notion of approximation on $\r$. 

If $p \in \lg$ is parabolic, then so is $g(p) \in \lg$ for every $g \in \Gamma$. Moreover, the set of all parabolic fixed points decomposes into a countable union of such orbits, so that we may choose a set $P$ of mutually inequivalent parabolic points and characterize the set of parabolic points as $\Gamma(P) = \{ g(p) : g \in \Gamma, p \in P\}$. We can now phrase the main definition of this section:
\begin{Def} A limit point $x \in \lg \subset S^n$  is called {\it badly approximable by a parabolic} $p \in P$  if there exists a constant $k(x,p) > 0$ such that $$|x - g(p) | > k(x,p) (1 - |g(0)|) \text{ for all } g \in \Gamma.$$ We denote this set by $BA(\Gamma, p)$. \\
We define the set of {\it badly approximable} limit points, written $BA(\Gamma)$, to be the set of limit points which are simultaneously badly approximable by all parabolics, i.e. $$BA(\Gamma) = \bigcap_{p \in P}BA(\Gamma, p).$$
\end{Def}

We are interested in the case when $BA(\Gamma)$ is absolutely winning in $\lg$. Note that by Lemma \ref{intersection}, it suffices to show that $BA(\Gamma,p)$ is absolutely winning for each $p \in P$. To apply our main theorem, we need two conditions, namely that the limit set $\lg$ is uniformly perfect and that there exists a $\Gamma$-invariant collection of disjoint horoballs based at the parabolic fixed points of $\Gamma$. For Kleinian groups, the former condition is a well-known consequence of being finitely generated, see \cite{JV}.
\begin{Thm} \label{unifperfect}
For $\Gamma$ a non-elementary finitely generated Kleinian group, $\Lambda(\Gamma)$ is uniformly perfect in $S^n$.
\end{Thm}

As for the latter condition, the existence of a disjoint $\Gamma$-invariant collection of horoballs, it is known to fail for some non-elementary Kleinian groups. In \cite{A}, Apanasov constructs a non-elementary infintely generated Kleinian group for which there is no such $\Gamma$-invariant collection. And while his proof crucially uses the infinite generation, it is our conjecture that there are finitely generated Kleinian groups for which this condition fails as well (for some discussion of this condition in the literature, see \cite{G}). Nonetheless, if such a collection exists, then we have 
\begin{Cor} Let $\Gamma$ be a non-elementary, finitely generated Kleinian group and $\{ H_j \}$ a $\Gamma$-invariant disjoint collection of horoballs based at the parabolic fixed points of $\Gamma$. Then $BA(\Gamma)$ is absolutely winning in $\lg$. 
\end{Cor}
\begin{proof} By Theorem \ref{main}, we know that the set $$\{x \in S^n : \exists k(x,p) > 0 \text{ such that } \forall g\in \Gamma, |x - g(p)| > k(x,p) r_{g(p)} \}$$ is absolutely winning in $\lg$. Here $r_{g(p)}$ is the Euclidean radius of the horoball based at $g(p)$, which is comparable to its radius in the visual metric on $S^n$. By the discussion in \cite{N}, there exists a constant $C$ depending only on $\Gamma$ and $p$ such that $r_{g(p)} > C (1 - |g(0)|)$. Therefore, the above set is a subset of $BA(\Gamma, p)$, which is therefore also absolutely winning in $\lg$. As mentioned before, since $P$ is a countable set, by Lemma \ref{intersection}, $BA(\Gamma)$ is absolutely winning in $\lg$.
\end{proof}

Despite our earlier discussion, there is a well-known condition on $\Gamma$ which guarantees the existence of such a collection. We say that $\Gamma$ is {\it geometrically finite} if it is finitely generated and its limit set is a disjoint union of conical and parabolic limit points. Recall that a limit point $x \in \lg$ is {\it conical} if there exists a sequence $\gamma_i \in \Gamma$ such that $\gamma_i(o)$ converges to $x$ in a cone with vertex at $x$.  This is equivalent to the existence of a constant $k=k(x)$ such that $|x - \gamma_i(o)| < k(1 - |\gamma_i(o)|)$.  The discussion in \cite{Bow} implies that for such groups, we can choose horoballs contained in the cuspidal ends of the quotient manifold $\h^{n+1}/\Gamma$ which satisfy the hypotheses of the corollary. Therefore, we immediately obtain
\begin{Cor} Let $\Gamma$ be a non-elementary geometrically finite Kleinian group. Then $BA(\Gamma)$ is absolutely winning in $\lg$. 
\end{Cor}

As mentioned in the introduction, this corollary can be viewed as a generalization of a result of McMullen which asserts this result holds for lattices. A Kleinian group $\Gamma$ is called a {\it lattice} if the quotient $\h^{n+1}/\Gamma$ carries a finite $\operatorname{Isom}(\h^{n+1})$-invariant measure. This condition is equivalent to $\Gamma$ being geometrically finite and $\lg = S^n$. 

It is remarked in \cite{S} that for a geometrically finite Kleinian group, a limit point $x$ is badly approximable if and only if any geodesic ray terminating at $x$ has bounded projection in the quotient manifold. The above corollary can therefore be restated as follows: for a non-elementary geometrically finite Kleinian group, the set of bounded geodesic rays is absolutely winning. In the next section we will show that winning subsets of $\lg$ have dimension equal to the critical exponent of $\Gamma$, thereby recapturing the result of Bishop and Jones \cite{BJ} in the geometrically finite case. 

As another consequence of winning, we also show:
\begin{Cor}  The set of bounded geodesics in $\h^{n+1}/\Gamma$ is winning in $\Lambda(\Gamma) \times \Lambda(\Gamma)$ for $\Gamma$ geometrically finite. 
\end{Cor}
\begin{proof} Indeed, $BA(\Gamma)$ consists of the set of limit points whose geodesic rays in $\h^{n+1}$ have bounded projection. Fixing a basepoint $o$ in $\h^{n+1}$ yields a canonical identification between $T_o \h^{n+1} \cong \partial \h^{n+1}$, given by $u \in T_o \h^{n+1} \mapsto u^+$, the endpoint of the unique geodesic ray generated by $u$. Since any two geodesic rays with the same endpoint are forward asymptotic, and hence lie within bounded distance of one another, they are both simultaneously bounded or unbounded in the projection to $\h^{n+1}/\Gamma$. Therefore, we have that $$\pi(u[0,\infty)) \mbox{ is bounded} \Leftrightarrow u^+ \in BA(\Gamma).$$ Clearly the projection of a bi-infinite geodesic $u$ is bounded if and only if the projections of its rays $u^+,u^-$ are, so the set of bounded geodesics may be identified with $$BA(\Gamma) \times BA(\Gamma) - \Delta,$$ where $\Delta$ is the diagonal.

Since $BA(\Gamma)$ is absolutely winning, by Lemma \ref{abswin} we have that $BA(\Gamma)$ is winning. Moreover, the product of two winning sets is winning with respect to the max metric on the product -- consider the projections, and employ the respective strategies and pull back the balls to obtain the next move -- so $BA(\Gamma) \times BA(\Gamma)$ is winning as a subset of $\Lambda(\Gamma) \times \Lambda(\Gamma)$. So it remains to show that we can exclude the diagonal and the set is still winning. 

It suffices for Alice to use her first move to ensure $A_1 \cap \Delta = \emptyset$. So let $B_1 = B((x,y),\rho) \subset \Lambda(\Gamma) \times \Lambda(\Gamma)$, where we may assume $\rho$ is sufficiently small. By Lemma 3.3, for $\beta$ sufficiently small, there exists $z \in \Lambda(\Gamma)$ such that $$B(z,\beta \rho) \subset B(y,\rho) - B(x,\beta \rho).$$ Set $A_1 = B((x,z),\beta \rho)$ and let $d(\cdot, \cdot)$ denote the spherical metric on $S^n$. We claim that $A_1 \cap \Delta = \emptyset$. If $(x',y') \in A_1$, then $$d(x, x') \leq \beta \rho$$ but $$d(x,y') \geq d(x,z) - d(y',z) > 2 \beta \rho - \beta \rho = \beta \rho$$ because $y' \in B(z,\beta \rho)$ and $d(z,x) > 2\beta \rho$. Hence $x' \ne y'$ as claimed.
\end{proof}

\end{section}

\begin{section}{Hausdorff Dimension}
We now turn to the question of computing the Hausdorff dimension of the set of badly approximable limit points of a Kleinian group as a consequence of winning. As mentioned in the introduction, it has been shown in \cite{Sch} that winning subsets of $\r^n$ have dimension $n$, and this result was generalized in \cite{F} to the support of any measure satisfying a power law. 
\begin{Lemma}\label{geom dim}(c.f. \cite{BFKRW}, Lemma $4.1$):\\
Let $K$ be the support of a measure $\mu$ satisfying a power law with exponent $\delta$, i.e. there exists a constant $C > 1$ such that for any ball $B(x,\rho)$ cenetered on $K$ with $\rho $ sufficiently small, $$C\inv \rho^{-\delta} \leq \mu(B(x,\rho)) \leq C \rho^{-\delta}.$$ Then $\text{dim}(S \cap K) \geq \delta$, whenever $S$ is winning on $K$. 
\end{Lemma}

Unfortunately, the global measure formula in \cite{SV} shows that in general the Patterson-Sullivan measure of a geometrically finite Kleinian group fails to satisfy such a power law. The obvious exception is that if $\Lambda(\Gamma) = S^n$, that is to say $\Gamma$ is of the first kind, then the Patterson-Sullivan measure is equivalent to Lebesgue measure, and therefore satisfies such a power law with $\delta = n$. Since the Hausdorff dimension of $S^n$ is $n$, we see that any winning subset in this case has full dimension, as noted previously. Additionally, in the case when $\Gamma$ has no parabolic fixed points, i.e. the convex cocompact case, it was shown in \cite{US} that the Patt
erson-Sullivan measure also satisfies a power law. 

However, Fishman, Simmons, and Urbanski \cite{FSU} have furnished the following lemma by adapting the construction used in Bishop and Jones \cite{BJ}. 
\begin{Lemma} Let $\Gamma$ be a Kleinian group, and let $\delta$ be its critical exponent. For every $\epsilon > 0$, there exists a uniformly perfect subset $K_\epsilon$ of the conical limit set which is the support of a Borel measure $\mu_\epsilon$ satisfying a power law with exponent $\delta - \epsilon$. 
\end{Lemma}

We can now furnish the proof of our claim in the introduction:
\begin{Cor} For a non-elementary geometrically finite Kleinian group $\Gamma$ with critical exponent $\delta$, the set of badly approximable points (equivalently the set of bounded geodesic rays) has dimension $\delta$. 
\end{Cor}
\begin{proof}The upper bound $\delta$ holds because $BA(\Gamma)$ is a subset of the conical limit set, whose Hausdorff dimension is at most $\delta$, \cite{N}. Conversely, the measure $\mu_\epsilon$ supported on $K_\epsilon$ satisfies a power law with exponent $\delta - \epsilon$, and hence by Lemma \ref{geom dim} any winning subset of $K_\epsilon$ has dimension at least $\delta - \epsilon$. However, because $K_\epsilon$ is diffuse in $\Lambda(\Gamma)$ and $BA(\Gamma)$ is absolutely winning on $\Lambda(\Gamma)$, by Theorem \ref{diffuse}, $BA(\Gamma) \cap K_\epsilon$ is absolutely winning as a subset of $K_\epsilon$. Therefore, by Lemma \ref{geom dim}, the dimension of $BA(\Gamma) \cap K_\epsilon$ is at least $\delta - \epsilon$, and hence the same holds for $BA(\Gamma)$. Letting $\epsilon \to 0$ completes the proof.
\end{proof}
\end{section}

\end{document}